\documentclass[a4paper,11pt]{amsart}
\usepackage{amsmath,amsthm,amssymb}
\usepackage[foot]{amsaddr}
\usepackage{enumerate}
\usepackage{algpseudocode}
\usepackage{algorithm}
\usepackage{fullpage}
\usepackage{tikz}
	\usetikzlibrary{calc}
	\usetikzlibrary{decorations.markings}
	\usetikzlibrary{patterns}
\usepackage[margin=0.5cm]{caption}

\newtheorem{lemma}{Lemma}
\newtheorem{theorem}[lemma]{Theorem}
\newtheorem{claim}[lemma]{Claim}

\newtheorem{corollary}[lemma]{Corollary}

\pgfdeclarelayer{background}
\pgfsetlayers{background,main}

\makeatletter
\renewcommand{\ALG@name}{Procedure}

\newcommand{\BB}{\mathcal{B}}
\newcommand{\ZZ}{\mathcal{Z}}

\newcommand{\cl}{\text{cl}}

\newcommand{\MM}{\mathbb{M}}

\newcommand{\HH}[1]{\mathcal{H}\!\left(#1\right)}

\newcommand{\e}{\text{e}}

\newcommand{\defeq}{:=}

\newcommand{\ignore}[1]{}

\date{\today}

\newcommand\AlgInput[1]{%
\Statex\hspace*{-\leftmargin}\textbf{Input: } #1%
}

\newcommand\AlgOutput[1]{%
\Statex\hspace*{-\leftmargin}\textbf{Output: } #1%
}

\newcommand\AlgLine{%
	\vspace*{-.7\baselineskip}\Statex\hspace*{\dimexpr-\leftmargin\relax}\rule{\textwidth}{0.4pt}%
}

\title{On the number of matroids compared to the number of sparse paving matroids}

\author{Rudi Pendavingh}

\author{Jorn van der Pol}
\address{Eindhoven University of Technology, Eindhoven, the Netherlands}
\email{\{R.A.Pendavingh,J.G.v.d.Pol\}@tue.nl}
\thanks{This research has been supported by the Netherlands Organisation for Scientific Research (NWO) grant 613.001.211.}

\begin{document}

\begin{abstract} It has been conjectured that sparse paving matroids will eventually predominate in any asymptotic enumeration of matroids, i.e. that $\lim_{n\rightarrow\infty} s_n/m_n = 1$, where $m_n$ denotes the number of matroids on $n$ elements, and $s_n$ the number of sparse paving matroids.
In this paper, we show that 
$$\lim_{n\rightarrow \infty}\frac{\log s_n}{\log m_n}=1.$$
We prove this by arguing that each matroid on $n$ elements has a faithful description consisting of a stable set of a Johnson graph together with a (by comparison) vanishing amount of other information, and using that stable sets in these Johnson graphs correspond one-to-one to sparse paving matroids on $n$ elements.

As a consequence of our result, we find that for some $\beta > 0$, asymptotically almost all matroids on $n$ elements have rank in the range $n/2 \pm \beta\sqrt{n}$.\end{abstract}

\maketitle

\section{Introduction}

After matroids up to 8 elements were enumerated by Blackburn, Crapo, and Higgs~\cite{BlackburnCrapoHiggs1973}, it was noted that a substantial fraction of the matroids were paving matroids, that is, matroids $M$ whose circuits are all of cardinality at least $r(M)$. Crapo and Rota speculated that perhaps `paving matroids will predominate in any enumeration of matroids'~\cite[p.\ 3.17]{CrapoRotaBook}. In~\cite[Conjecture 1.6]{MayhewNewmanWelshWhittle2011} Mayhew, Newman, Welsh and Whittle make the more precise conjecture in that the asymptotic fraction of matroids on $n$ elements that are paving tends to 1 as $n$ tends to infinity. Their conjecture is equivalent to the seemingly stronger statement that
\begin{equation}\label{conj:all_sparse}\lim_{n\rightarrow\infty} s_n/m_n = 1.\end{equation}
Here $m_n$ denotes the number of matroids on a fixed ground set of $n$ elements, and $s_n$ is the number of sparse paving matroids (a matroid is sparse paving if both it and its dual are paving).

Sparse paving matroids seem benign objects compared to matroids in general. For example, it is straightforward that a sparse paving matroid on sufficiently many elements is highly connected. The predominance of sparse paving matroids as in \eqref{conj:all_sparse}  would thus immediately imply the predominance of $k$-connected matroids, which is conjectured but remains an open problem. Similarly, it is relatively straightforward that asymptotically all sparse paving matroids have a fixed uniform matroid $U_{a,b}$ as a minor, but the analogous statement for general matroids is open. Further examples along these lines are easy to find, whence the interest in conjecture \eqref{conj:all_sparse}.

Combining the lower bound of Graham and Sloane~\cite[Theorem 1]{GrahamSloane1980} (as pointed out in~\cite{MayhewWelsh2013}) on $s_n$ and the upper bound of Bansal, Pendavingh and van~der~Pol~\cite{BansalPendavinghVanderpol2014} on $m_n$, we have
\begin{equation}\label{eq:upper_lower}
	\frac{1}{n}\binom{n}{\lfloor n/2\rfloor}\leq\log s_n\leq \log m_n\leq \frac{2+o(1)}{n}\binom{n}{\lfloor n/2\rfloor} \qquad\text{ as }n\rightarrow\infty.
\end{equation}
These bounds do not suffice to prove \eqref{conj:all_sparse}, but merely imply
$$ \log s_n\leq \log m_n\leq (2+o(1)) \log s_n\text{ as }n\rightarrow \infty,$$
or equivalently, that $s_n\leq m_n\leq s_n^{2+o(1)}$. 

The sparse paving matroids showed a somewhat less benign side when we attempted to narrow the gap between the upper and the lower bound in \eqref{eq:upper_lower}. Being unable to improve either bound, we devised a way to directly compare the number of matroids to the number of sparse paving matroids. This enabled us to prove the main result of this paper, that 
$$\log m_n\leq (1+o(1)) \log s_n\text{ as }n\rightarrow \infty,$$
or equivalently, $m_n=s_n^{1+o(1)}$.
Our method is closely related to the one used in \cite{BansalPendavinghVanderpol2014} to prove the upper bound on $m_n$, which itself is an adaptation of a method to bound $s_n$.


We will briefly outline the method and describe how it differs from earlier work. Key to our method is an algorithm for producing a compressed description of any given matroid $M$ on $E$ of rank $r$. The compression algorithm considers the set of bases of $M$ as a subset of all the $r$-subsets of $E$, which are the vertices the Johnson graph $J(E,r)$. We obtain a compact description of the matroid by starting from the full set of vertices $A$ of the Johnson graph $G=J(E,r)$ and iteratively taking away neighborhoods of vertices from $A$ while describing the set  of bases among these neighbourhoods. As long as there are vertices of high degree in $G[A]$ to pick, the rate at which we need to add information into our matroid description compares favourably to the decrease in the size of $A$.  We argued that while $A$ is large there will be such vertices of high degree, and by the time~$A$ contains no more than a certain~$\alpha$-fraction of the vertices, the total amount of information stored so far will still be relatively modest. In our previous paper, we completed the description of the matroid by adding  $\alpha \binom{n}{r}$ bits to describe the subset of bases among the remaining vertices of $A$, and this is what ultimately dominated the length of the matroid description we obtained. Hence, the cost of describing the bases among the final set $A$ was the bottleneck for producing a tighter upper bound on $m_n$.

Previously, the presence of a vertex of high degree in~$G[A]$ was necessary to show that the set of bases in its neighbourhood can be described using a relatively small amount of information.
In the present paper, we show that in the final stage a small maximum degree is advantageous as well. In particular, if~$G[A]$ does not contain vertices of high degree, then most of the neighbours of~$X \in A$ lie outside $A$, so that for most of these neighbours it is known whether they are a basis of the matroid or not from the matroid description so far.
Exploiting this information and matroid structure, we find that the $\alpha \binom{n}{r}$ bits we used before can be replaced by a certain stable set $T\subseteq A$ to obtain a faithful description of the matroid. Since our matroid description now consists of a stable set in the Johnson graph together with some contained amount of further information, it becomes possible to compare the number of matroids directly to the number of stable sets in the Johnson graph. As stable sets in $J(E,r)$ are in 1-1 correspondence to sparse paving matroids, this implies our main result.

A further result in this paper is that there exists~$\beta>0$ such that asymptotically all matroids on $n$ elements have a rank between~$n/2-\beta\sqrt{n}$ and $n/2+\beta\sqrt{n}$. This is related to a second conjecture from~\cite[Conjecture 1.10]{MayhewNewmanWelshWhittle2011}, that asymptotically all matroids on $n$ elements have a rank between~$(n-1)/2$ and~$(n+1)/2$.

After giving preliminaries on graphs and matroids in Section 2, we present both results in Section 3. Finally, we discuss several remaining open problems related to our main results in Section 4.

\section{Preliminaries}

\subsection{Graphs and stable sets} We only consider loopless,  undirected graphs in this paper. 
If $G$ is any graph, then we write $\Delta(G)$ for the maximum degree in $G$. Further, for any $A \subseteq V(G)$, we write $G[A]$ for the subgraph of $G$ induced by the vertices in $A$. A set of vertices $A\subseteq V(G)$ is {\em stable} if $G[A]$ spans no edges. We write $i(G)$ for the number of stable sets in $G$.

In~\cite{BansalPendavinghVanderpol2014}, Nikhil Bansal and the current authors proved the following result on the number of stable sets in regular graph.

\begin{theorem}\label{thm:stablesets}
	Let $G$ be a $d$-regular graph on $N$ vertices with smallest eigenvalue $-\lambda$, with $d > 0$. Then $i(G) \le \sum_{s =0}^{\lceil \sigma N\rceil} \binom{N}{s} 2^{\alpha N}$, where $\alpha = \frac{\lambda}{d+\lambda}$ and $\sigma = \frac{\ln(d+1)}{d+\lambda}$.
\end{theorem}

The quantity $\alpha N$ is known as the \emph{Hoffman bound}, which is an upper bound on the cardinality of a stable set in $G$.

\subsection{The Johnson graph}

Let $E$ be a finite set, and let $0 < r < |E|$. The Johnson graph~$J(E,r)$ is the graph with vertex set $$\binom{E}{r}:=\{X\subseteq E: ~|X|=r \},$$ in which any two vertices are adjacent if and only if they have $r-1$ elements in common; equivalently, the vertices $X$ and $Y$ are adjacent whenever $|X \triangle Y| = 2$.

Let $[n]:=\{1,\ldots, n\}$. We abbreviate $J(n,r) := J([n],r)$. Clearly~$J(E,r) \cong J(n,r)$ when $|E|=n$. Note that $J(n,r) \cong J(n,n-r)$; the function $X \mapsto [n] \setminus X$ provides an explicit isomorphism.

The spectrum of the Johnson graph is known: if $r \le n/2$, the eigenvalues of $J(n,r)$ are
\begin{equation}
	(r-i)(n-r-i)-i, \qquad i = 0,1,\ldots,r.
\end{equation}
Hence, the smallest eigenvalue of $J(n,r)$ is $-r$, and using Theorem~\ref{thm:stablesets} it was derived in \cite{BansalPendavinghVanderpol2014} that
\begin{equation}\label{eq:Johnson_stable}
	\log i(J(n,r)) \le \frac{2}{n}\binom{n}{\lfloor n/2\rfloor}(1+o(1)) \text{ as } n\rightarrow \infty.
\end{equation}
A corresponding lower bound is obtained from a construction by Graham and Sloane~\cite[Theorem 1]{GrahamSloane1980}, who show that
\begin{equation}\label{eq:Johnson_stable2}
	\log i(J(n,r)) \ge \frac{1}{n}\binom{n}{r}.
\end{equation}

If $X \in \binom{E}{r}$, then we will use the graph-theoretic term \emph{neighbourhood} to denote the set
\begin{equation*}
	N(X) := \{X-x+y : x \in X, y \in E\setminus X\}.
\end{equation*}
Note that $N(X)$ is precisely the neighbourhood of $X$, seen as a vertex in the Johnson graph $J(E,r)$.
The neighbourhood of $X$ has the structure of a Cartesian graph product of $K_r$ and $K_{n-r}$. In particular, we can distinguish `rows'
\begin{align}
	&R_X(x) \defeq \{X-x+y : y \in E\setminus X\}, && x \in X \\
	\intertext{and `columns'}
	&C_X(y) \defeq \{X-x+y : x \in X\}, && y \in E\setminus X,
\end{align}
that induce cliques in the neighbourhood of $X$, see Figure~\ref{fig:neighbourhood}.
\begin{figure}
	\begin{minipage}[b]{0.49\textwidth}\centering
		\begin{tikzpicture}[scale=0.7,font=\normalsize,inner sep=2.8]
			\begin{scope}[scale=1.5]
				\foreach\y in {0,2,3} {
					\foreach\x in {0, 1, 2, 4} {
						\node[circle,fill=black] (N-\x-\y) at (\x,\y) {};
					}
					\node[font=\huge] at (3,\y) {$\cdots$};
				}
				\foreach\x in {0,1,2,4} {
					\node[font=\Huge] at (\x,1) {$\vdots$};
				}
				
				\draw ($(N-0-3) + (-0.4,0.6)$) rectangle ($(N-0-0)+(0.4,-0.6)$);
				\node[rotate=-30,anchor=north west] at ($(N-0-0)+(0,-0.6)$) {$C_X(y)$};
				\draw ($(N-2-3) + (-0.4,0.6)$) rectangle ($(N-2-0)+(0.4,-0.6)$);
				\node[rotate=-30,anchor=north west] at ($(N-2-0)+(0,-0.6)$) {$C_X(y')$};
				\draw ($(N-0-3) + (-0.6,0.4)$) rectangle ($(N-4-3)+(0.6,-0.4)$);
				\node[anchor=west, inner sep=6pt] at ($(N-4-3)+(0.6,0)$) {$R_X(x)$};
			\end{scope}
			
			
			\draw (N-0-3) + (0,2) node[] (N-0-3-label) {$X-x+y$};
			\path[decoration={markings,mark=at position 0.99 with {\arrow[scale=2, transform shape]{>}}}, postaction={decorate}]
				(N-0-3-label.south) to[in=30,out=-30, shorten >=2] (N-0-3.north east);
			\path (N-0-3-label.south) edge[in=30,out=-30, shorten >=2] (N-0-3.north east);
			
			\draw (N-2-3) + (1,2) node[] (N-2-3-label) {$X-x+y'$};
			\path[decoration={markings,mark=at position 0.99 with {\arrow[scale=2, transform shape]{>}}}, postaction={decorate}]
				(N-2-3-label.south) to[in=30,out=-30, shorten >=2] (N-2-3.north east);
			\path (N-2-3-label.south) edge[in=30,out=-30, shorten >=2] (N-2-3.north east);
		\end{tikzpicture}
	\end{minipage}
	\begin{minipage}[b]{0.49\textwidth}\centering
		\begin{tikzpicture}[scale=0.7, font=\normalsize, inner sep=2.8]
			\begin{scope}[scale=1.5]
				\foreach\y in {0,2,3} {
					\foreach\x in {0, 1, 2, 4} {
						\node[circle, draw=black, ultra thick] (N-\x-\y) at (\x,\y) {};
					}
					\node[font=\huge] at (3,\y) {$\cdots$};
				}
				\foreach\x in {0,1,2,4} {
					\node[font=\Huge] at (\x,1) {$\vdots$};
				}
				
				\foreach\x in {0,1,4} {
					\node[circle,fill=black] at (\x,2) {};
				}
				
				\begin{pgfonlayer}{background}
				\foreach \x/\ind in {0/1,1/1,2/0,4/1} {
					\ifnum\ind=0
						\draw[pattern=north west lines, pattern color=gray] ($(N-\x-3) + (-0.4,0.6)$) rectangle ($(N-\x-0)+(0.4,-0.6)$);
					\else
						\draw ($(N-\x-3) + (-0.4,0.6)$) rectangle ($(N-\x-0)+(0.4,-0.6)$);
					\fi
				}

				\node[rotate=-30,anchor=north west] at ($(N-0-0)+(0,-0.6)$) {$y \in D$};
				\node[rotate=-30,anchor=north west] at ($(N-2-0)+(0,-0.6)$) {$y' \not\in D$};

				\foreach \y/\ind in {0/0, 2/1, 3/0} {
					\ifnum\ind=0
						\draw[pattern=north west lines, pattern color=gray] ($(N-0-\y) + (-0.6,0.4)$) rectangle ($(N-4-\y)+(0.6,-0.4)$);
					\else
						\draw ($(N-0-\y) + (-0.6,0.4)$) rectangle ($(N-4-\y)+(0.6,-0.4)$);
					\fi
				}
				
				\node[anchor=west, inner sep=6pt] at ($(N-4-3)+(0.6,0)$) {$x \not\in C$};
				\end{pgfonlayer}
			\end{scope}
		\end{tikzpicture}
	\end{minipage}
	\begin{minipage}[t]{0.49\textwidth}
		\captionof{figure}{\label{fig:neighbourhood}The neighbourhood of $X$ in the Johnson graph. The rows and columns, indexed by $x \in X$ resp.\ $y \in E\setminus X$, form cliques.}
	\end{minipage}
	\begin{minipage}[t]{0.49\textwidth}
		\captionof{figure}{\label{fig:neighbourhood_bases}If $r(X) = r-1$, then the bases in $N(X)$ (represented here by solid vertices) can be recovered by removing all rows corresponding to $x \not\in C$ and columns corresponding to $y \not\in D$.}
	\end{minipage}
\end{figure}
\subsection{Matroids} A matroid is a pair $M=(E,\BB)$ such that $E$ is a finite set and $\BB$ is a non-empty set of subsets of $E$ satisfying the {\em base exchange axiom} 
\begin{equation}\text{for all } B, B'\in \BB\text{ and }e\in B\setminus B'\text{ there exists an }f\in B'\setminus B\text{ such that } B-e+f\in\BB.\end{equation}
We assume familiarity with the definitions of circuit, flat, rank, dual etc., for which we refer to Oxley's book \cite{OxleyBook}. 

While we mostly use the notation of \cite{OxleyBook}, the following is nonstandard. We write $\MM_{n,r}$ for the set of all matroids of rank $r$ with ground set $[n]$, and $\MM_n$ for the set of matroids with ground set $[n]$, and we put $m_{n,r}:=\#\MM_{n,r}$ and $m_{n}:=\#\MM_{n}$. 

A matroid $M$ is said to be {\em paving} if each circuit of $M$ has cardinality at least $r(M)$, and $M$ is {\em sparse paving} if both $M$ and its dual are paving. 
We define $$s_{n,r}:=\#\{ M\in \MM_{n,r}: M\text{ is sparse paving}\}, ~s_n:=\#\{ M\in \MM_n: M\text{ is sparse paving}\}.$$

The following lemma is essentially established by Piff and Welsh~\cite{PiffWelsh1971}. A proof can be found in~\cite[Lemma 8]{BansalPendavinghVanderpol2014}.
\begin{lemma}\label{prop:johnson_sparsepaving}
	Let $\BB \subseteq \binom{E}{r}$, then $\BB$ is the collection bases of a sparse paving matroid if and only if $\binom{E}{r}\setminus\BB$ is a stable set	in $J(E,r)$.
\end{lemma}
Hence $s_{n,r}=i(J(n,r))$, which is bounded by \eqref{eq:Johnson_stable}.

\subsection{The local structure of matroids in the Johnson graph}
The following matroid lemma is elementary, but it has a central role in this paper.  
\begin{lemma}\label{lemma:circuit_cocircuit}
Let $M=(E, \BB)$ is a matroid of rank $r$, and let  $X\in \binom{E}{r}$ be such that $r_M(X)=r-1$. There is a unique circuit $C$ of $M$ so that $C\subseteq X$ and a unique cocircuit $D$ of $M$ so that $D\cap X=\emptyset$. \end{lemma}
The lemma implies that if $r(X)=r-1$, then the set of bases of $M$ in the neighbourhood of $X$ has a very simple structure that is completely determined by the circuit $C$ and the cocircuit $D$:
\begin{equation}\label{eq:simple_neighborhood}
N(X)\cap \BB=\{X-x+y: x\in C, y\in D\},
\end{equation}
see Figure~\ref{fig:neighbourhood_bases}. It is this simplicity which enables us to make faithful descriptions of matroids which are nearly as concise as a description of a stable set in the Johnson graph. In the present paper, we use \eqref{eq:simple_neighborhood} in two ways. 

The first use was already implicit in our previous paper \cite{BansalPendavinghVanderpol2014}, where we used {\em local covers} as a short certificate for (in-)dependence in the neighbourhood of a non-basis. We review the basic definition, and quote the lemma which we will again use in our present argument.



If $Y$ is a set in a matroid $M$, and $F \in \mathcal{F}(M)$ is a flat such that $|F \cap Y| > r_M(F)$, then $Y$ is necessarily dependent, and $(F,r_M(F))$ serves as a certificate for the dependence of $Y$. In this case, we say that $(F,r_M(F))$ \emph{covers} $Y$. Note that $M$ can be reconstructed from its groundset, rank, and a collection of (flat,rank)-pairs such that each non-basis of $M$ is covered by at least one flat in the collection.\protect\footnote{NB: In~\cite{BansalPendavinghVanderpol2014}, just the flat $F$ (rather than the pair $(F,r_M(F))$) was used as a certificate for dependence. However, as the rank of $F$ is necessary to reconstruct the matroid, 
we choose to use the pair here.
}

A \emph{local cover} at $X \in \binom{E}{r}$ is a subset $\mathcal{Z}_X \subseteq \{(F,r_M(F)) : F \in \mathcal{F}(M)\}$ with the property that each $Y \in N(X) \cup\{X\}$ is either independent, of covered by some $(F,r_M(F)) \in \mathcal{Z}_X$. If $X$ is a non-basis, then the local cover can be surprisingly small, as the following lemma shows. A proof can be found in \cite{BansalPendavinghVanderpol2014}, and also follows from Lemma \ref{lemma:circuit_cocircuit}.

\begin{lemma}[{\cite[Lemma 20]{BansalPendavinghVanderpol2014}}]\label{lemma:local_cover}
	Let $M$ be a rank-$r$ matroid on groundset $E$, and let $X \in \binom{E}{r}$ be dependent in $M$. Then there exists $\mathcal{Z}_X \subseteq \mathcal{F}(M)$ with $|\mathcal{Z}_X| \le 2$, that covers each non-basis $Y \in N(X) \cup \{X\}$.
\end{lemma}
\ignore{\proof If $r_M(X)<r-1$, then taking $F=\cl_M(X)$ we have $|F\cap Y|\geq |X\cap F|=r-1>r_M(X)=r_M(F)$ for all $Y\in N(X)$, so that $\mathcal{Z}_X:=\{F\}$ will do. 
If $r_M(X)=r-1$, then  by the Lemma there is a unique circuit $C$ contained in $X$ and unique cocircuit $D$ disjoint from $X$. Then $\mathcal{Z}_X:=\{\cl_M(C), E\setminus D\}$ is a local cover, since if  $Y\in N(X)$ is dependent, then either $C\subseteq Y$ and $\cl_M(C)$ covers $Y$, or $D\cap Y=\emptyset$ and $E\setminus D$ covers $Y$.
\endproof}

The second use of \eqref{eq:simple_neighborhood} is new to this paper. The very restricted structure of $N(X)\cap \BB$ in the neighbourhood of a dependent set $X$ will allow us to recover the partition $(N(X)\setminus \BB, N(X)\cap \BB)$ from partial information  $(K\setminus \BB, K\cap \BB)$ for certain $K\subseteq N(X)$, so that a faithful matroid encoding can be even more sparse if we rely on a decoder which can infer from \eqref{eq:simple_neighborhood}.

\subsection{Binomial coefficients}

We will use the following standard bound on the sum of binomial coefficients
\begin{equation}\label{eq:binomial_sum}
	\sum_{i=0}^k \binom{n}{k} \le \left(\frac{\e n}{k}\right)^k,
\end{equation}
a proof of which can be found in~\cite[Proposition 1.4]{Jukna2011}. The following estimate of the central binomial coefficient is a consequence of Stirling's approximation:
\begin{equation}\label{eq:binomial_central}
	\binom{n}{\lfloor n/2\rfloor} = \Theta\left(\frac{2^n}{\sqrt{n}}\right)\qquad\text{as $n \to \infty$}.
\end{equation}
We will also need the following result on binomial coefficients.
\begin{lemma}\label{lemma:binomial4}
	If $k < \lfloor n/2\rfloor$, then $\binom{n}{\lfloor n/2 \rfloor -k} \le \binom{n}{\lfloor n/2\rfloor} \exp_2\left(-\frac{k^2}{\lceil n/2\rceil + k}\right)$.
\end{lemma}
\begin{proof}
	Using the identity $\binom{n}{t-1} = \frac{t}{n+1-t}\binom{n}{t}$ repeatedly, we find
	\begin{multline*}
			\binom{n}{\lfloor n/2\rfloor -k}
				= \binom{n}{\lfloor n/2\rfloor} \prod_{i=0}^{k-1} \frac{\lfloor n/2\rfloor - i}{n+1 - (\lfloor n/2\rfloor -i)} \\
				\le\binom{n}{\lfloor n/2\rfloor} \prod_{i=0}^{k-1} \left(1-\frac{k}{\lceil n/2\rceil + k + 1 - i}\right)
				\le \binom{n}{\lfloor n/2\rfloor} \left(1-\frac{k}{\lceil n/2\rceil + k + 1}\right)^k.
	\end{multline*}
	The lemma now follows from the inequality $1-x \le \e^{-x}$.
\end{proof}

\section{The number of matroids}

\subsection{A procedure for encoding non-bases}

The bound in Theorem~\ref{thm:stablesets} is based on a procedure to construct a concise description of stable sets in a regular graph. Bounding the number of such concise descriptions immediately gives an upper bound on the number of stable sets in a regular graph. The procedure was adapted from a procedure described by Alon, Balogh, Morris, and Samotij in~\cite{ABMS2012}, who cite Kleitman and Winston~\cite{KleitmanWinston1982} as the original source. A detailed account of the procedure for counting stable sets can be found in the recent survey paper by Samotij~\cite{Samotij2014}.

In~\cite{BansalPendavinghVanderpol2014}, this idea was combined with local covers to construct a concise description of matroids. In particular, it was shown that any matroid can be described by a stable set (in the Johnson graph) of non-bases $S$, a collection of flats covering all non-bases in $S \cup N(S)$, and 
a (relatively short) list of all the non-bases that are not yet covered.
By bounding the number of possibilities for each of these sets, one obtains an upper bound on the number of matroids.

The bound that was obtained in~\cite{BansalPendavinghVanderpol2014} is dominated by the list of non-bases that are not yet covered. This seems to be wasteful: if we can take into account more information about this list, we may be able to obtain stronger bounds. In the current section, we extend the encoding procedure to capture more information about this list of non-bases.

We will analyse the number of matroids on a fixed ground set $E = [n]$ of rank~$r$. In what follows, we will assume that $r \le n/2$, and we will use $G = J(E,r)$ as a shorthand notation.

We will further fix a linear ordering $\le_G$ on the vertices of $G$. By the \emph{canonical ordering} on $A \subseteq V(G)$, we refer to the following procedure to order the set $A$ linearly. Let $v$ be the vertex with maximum degree in $G[A]$; if there are multiple such $v$, then we choose the one that is minimal with respect to $\le_G$. Call $v$ the first vertex in the canonical ordering, and apply iteratively to $A\setminus\{v\}$.

Throughout the main loop of the procedure, two disjoint vertex sets are maintained: a set $S$ of selected vertices (which will grow during execution), and a set $A$ of available vertices (which will shrink). The main loop runs as long as~$|A|$ is sufficiently large, or~$G[A]$ contains a vertex of sufficiently high degree.

\begin{algorithm}[t]
\begin{algorithmic}[0]
\AlgInput{Matroid $M=(E,\BB)$ of rank $r$ on $n$ elements, $r \le n/2$}
\AlgOutput{$(S,\mathcal{Z},A,T)$}
\AlgLine
\State Set $A\gets V(G)$, $S\gets\emptyset$, $\ZZ\gets\emptyset$
\Comment{$G = J(E,r)$}
\While{$|A| > \alpha_{n,r} N$ \textbf{or} $\Delta(G[A]) \ge r$} \Comment{$\alpha_{n,r} N$ is the Hoffman bound}
	\State Pick the first vertex $X$ in the canonical ordering of $A$
	\If{$X$ is dependent in $M$}
		\State Set $S\gets S\cup\{X\}$, $A\gets A\setminus(\{X\} \cup N(X))$
		\State Set $\ZZ\gets \ZZ\cup\ZZ_X$
			\Comment{$\ZZ_X$ defined in Lemma~\ref{lemma:local_cover}}
	\Else
		\State Set $A\gets A\setminus\{X\}$
	\EndIf
\EndWhile
\Statex
\State Set $A'\gets \{X \in A\setminus\BB : \exists e \in X, f \in E\setminus X \text{ such that}$
\State \hfill$R_X(e) \cap A = C_X(f) \cap A = \emptyset, \quad R_X(e) \cap\BB, C_X(f) \cap\BB \neq \emptyset\}$
\State Set $T\gets$ maximal stable set in $G[A']$
\end{algorithmic}
\caption{\label{proc:encoding}The procedure for encoding matroids}
\end{algorithm}

\subsection{Analysis of the procedure}

Throughout this section, we write
\begin{equation}
	N = \binom{n}{r},\quad\text{resp.}\quad d = r(n-r),
\end{equation}
for the number of vertices, resp.\ degree, of the Johnson graph $J(n,r)$. We will also abbreviate
\begin{equation}
	\alpha_{n,r} = \frac{d}{d+\lambda} = \frac{1}{n-r+1},
\end{equation}
and
\begin{equation}
	\sigma_{n,r} = \frac{\ln(d+1)}{d+\lambda} = \frac{\ln(r(n-r)+1)}{r(n-r+1)}.
\end{equation}
The quantities~$\alpha_{n,r}$ and~$\sigma_{n,r}$ will play the same role as~$\alpha$ and~$\sigma$ in the statement of Theorem~\ref{thm:stablesets}. In particular, $\alpha_{n,r} N$ is the Hoffman bound.

\begin{claim}\label{claim:output_bound}
	Upon termination of the procedure, we have $|S| \le \left(\sigma_{n,r} + \frac{1}{r+1}\alpha_{n,r}\right)\binom{n}{r}$, $|A| \le \alpha_{n,r}\binom{n}{r}$, and $|\mathcal{Z}| \le 2|S|$.
\end{claim}

For future reference, we record
\begin{equation}
	\tilde\sigma_{n,r} = \sigma_{n,r} + \frac{1}{r+1}\alpha_{n,r}.
\end{equation}

\begin{proof}
	Note that the sets $S$, $\mathcal{Z}$ and $A$ only change during execution of the while loop.

	In each traversal, $|A|$ decreases, and the procedure does not stop before $|A| \le \alpha_{n,r} N$, thus proving the bound on $|A|$.

	As $A$ only gets smaller in each traversal, execution of the while loop falls apart into two stages: during the first stage, $|A| > \alpha_{n,r} N$, while during the second stage, $|A| \le \alpha_{n,r} N$ and $\Delta(G[A]) \ge r$.

	The first stage was analysed in~\cite[Lemma 16]{BansalPendavinghVanderpol2014}, where it was shown that during this stage at most $\sigma_{n,r} N$ vertices are added to $S$. At the start of the second stage, $A$ contains at most $\alpha_{n,r} N$ vertices. Throughout this stage, each element that is added to $S$ has degree at least $r$ in $A$, as they are the first vertex in the canonical ordering on $A$. So each time a vertex is added to $S$ during the second stage, at least $r+1$ vertices are removed from $A$. Hence, during the second stage, at most $\frac{1}{r+1}\alpha_{n,r} N$ vertices are added to $S$. Combining the bounds on the number of elements added to $S$ during both stages, we obtain the bound on $|S|$.

	The set $\mathcal{Z}$ is only extended when a vertex is added to $S$. Each time this happens, at most two new flats are introduced to $\mathcal{Z}$ by Lemma~\ref{lemma:local_cover}, so $|\mathcal{Z}| \le 2|S|$.
\end{proof}

The following claim is obvious.

\begin{claim}\label{claim:small_degree}
	Upon termination of Stage 2, all vertices in~$A$ have at most~$r-1$ neighbours in $A$.
\end{claim}

\begin{claim}\label{claim:stable}
	$S \cup T$ is a stable set in $J(n,r)$.
\end{claim}

\begin{proof}
	First, $S$ is a stable set, as each time a vertex is added to $S$, its neighbours are deleted from $A$, and hence never will be considered again. By the same argument, no element in $A$ has a neighbour in $S$. By construction, the set $T$ is stable, and as $T\subseteq A$, it follows that $S\cup T$ is stable.
\end{proof}

\begin{claim}\label{claim:S-determines-A}
	Upon termination of the procedure, the triple $(S,T,A)$ is completely determined by $S\cup T$.
\end{claim}

\begin{proof}
	Let $(S,\mathcal{Z},A,T)$ be the output of the procedure when it is run on input $M = ([n],\BB)$ or rank $r$. As $S\cup T$ is a stable set in $J(n,r)$, the set $\binom{[n]}{r} \setminus (S\cup T)$ is the set of bases of a sparse paving matroid $M'$. Let $(S',\mathcal{Z}',A',T')$ be the result of running the procedure on input $M'$. We claim that $S' = S$ and $A' = A$. This implies the lemma, as $T = (S\cup T)\setminus S'$.

	The claim follows, as the order in which $r$-sets are considered in the main loop is deterministic, and depends only on the choice that is made in each traversal. These choices are the same in both instances. By construction of $M'$, we have $X$ dependent (when encoding $M$) if and only if $X \in S$, which is equivalent to saying that $X$ is dependent (when encoding $M'$). The final equivalence follows, since vertices in $T$ come after vertices in $S$ in the canonical ordering in each traversal.
\end{proof}

%
%

Let $K$ be the set of non-bases of a matroid. Note that throughout the procedure,
\begin{equation}\label{eq:invariant}
	S \subseteq K \subseteq S \cup N(S) \cup A
\end{equation}
is maintained as an invariant.

If the encoding procedure indeed constructs a concise description of matroids, it should be possible to reconstruct $K$ from the output of the procedure. Non-bases in $S \cup N(S)$ are easily recognised, as they are covered by $\mathcal{Z}$. On the other hand, recognising non-bases in $A$ is a bit more involved.

The following claim is the engine of the corresponding decoding procedure. It roughly states that if $X \in A$, then $X$ being dependent is completely determined by $(S,\mathcal{Z},T)$.

\begin{claim}\label{claim:reconstruction}
	Let~$M$ be a matroid without loops and coloops, and let~$(S,\mathcal{Z},A,T)$ be the output of the procedure on input $M$. Let $X \in A$, then $X$ is dependent in~$M$ if and only if
	\begin{enumerate}[(i)]
		\begin{item}
			there exists $x \in X$ such that $R_X(x)$ is disjoint from $A$ and fully dependent; or
		\end{item}
		\begin{item}
			there exists $y \in E\setminus X$ such that $C_X(y)$ is disjoint from $A$ and fully dependent; or
		\end{item}
		\begin{item}
			$X \in T$; or
		\end{item}
		\begin{item}
			$X$ has a neighbour $Y = X - x + y$ in $T$, and there are $e \in Y$, $f \in E \setminus Y$ with the property that both $R_Y(e)$ and $C_Y(f)$ are disjoint from $A$, both contain a basis, and at least one of $Y-e+x \in R_Y(e)$ and $Y-y+f \in C_Y(f)$ is dependent.
		\end{item}
	\end{enumerate}
\end{claim}

\begin{proof}
	To prove sufficiency, suppose that $X \in A$. If (i) holds, then $X$ must be dependent. For if $X$ would be independent, and each $Y \in R_X(x)$ is dependent, then $x$ is a coloop, which contradicts our assumption on $M$. Similarly, if (ii) holds, then $X$ must be dependent. For if it would be independent, and each $Y \in C_X(y)$ would be dependent, then $y$ is a loop, again contradicting our assumption on $M$.

	If (iii) holds, then $X$ must be dependent, as by construction $T$ contains only non-bases.

	Finally, suppose that (iv) holds. Let $Y = X - x + y$ be an element of $T$ neighbouring $X$ satisfying the properties mentioned in (iv). As $X$ is dependent, and has an independent neighbour, it must have rank $r-1$. It follows that there is a unique circuit $C$ contained in $Y$, and a unique cocircuit $D$ disjoint from $Y$. As $C_Y(f)$ contain an independent set, we have
	\begin{equation*}
		C = \{g \in Y : Y - g + f\text{ is independent}\} \subseteq C_Y(f),
	\end{equation*}
	and since $R_Y(e)$ contains an independent set, we have
	\begin{equation*}
		D = \{h \in E\setminus Y : Y - e + h\text{ is independent}\} \subseteq R_Y(e).
	\end{equation*}
	Note that $X = Y - y + x$ is independent if and only if $C$ is not contained in $X$ (so $y \in C$, or equivalently $Y - y + f$ is independent), and $X$ is not disjoint from $D$ (so $x \in D$, or equivalently $Y - e + x$ is independent). Taking the contrapositive, we find that $X$ is dependent if and only if at least one of $Y-y+f$ or $Y-e+x$ is dependent.

	It remains to prove necessity. Let us assume that $X \in A$ is dependent. If neither (i) nor (ii) holds, then~$X \in A'$. As~$T$ is a maximal stable set in $G[A']$, we have $A'\subseteq T\cup N(T)$, so either $X \in T$, or $X$ has a neighbour in $T$. In the former case, we have (iii), and we are done. So assume that $X$ has a neighbour in $T$. Call this neighbour $Y = X - x + y$. As $T \subseteq A'$, there must exist $e \in Y$ and $f \in E\setminus Y$, so that $R_Y(e)$ and $C_Y(f)$ are disjoint from $A$, and both contain an independent set. Now we use again that $X$ is dependent if and only if at least one of $Y-y+f$ or $Y-e+x$ are dependent.
\end{proof}

The following lemma shows that the procedure can be used to construct an alternative description for a matroid, provided that the matroid has no loops or coloops.

\begin{lemma}\label{lemma:concise-description}
	Let $M = ([n],\BB)$ be a matroid of rank $r$ that does not have any loops or coloops. Let $(S,\mathcal{Z},A,T)$ be the output of the procedure on input $M$. Then $M$ can be reconstructed from $n$, $r$, and $(S\cup T, \mathcal{Z})$.
\end{lemma}

\begin{proof}
	First, it follows from Claim~\ref{claim:S-determines-A} that the pair $(S\cup T, \mathcal{Z})$ actually contains the more detailed information $(S, \mathcal{Z}, A, T)$. The matroid $M$ can be reconstructed from $n$, $r$, and $(S,\mathcal{Z},A,T)$ if for each $X \in \binom{[n]}{r}$ it can be decided whether $X$ is dependent or independent, bases on the available information alone.

	Let $K = \binom{[n]}{r}\setminus\BB$ be the set of non-bases in $M$. Recall that throughout the procedure, \eqref{eq:invariant} is maintained as an invariant. By Claim~\ref{claim:S-determines-A}, it can be verified from $(S,\mathcal{Z})$ if $X$ is in $S \cup N(S) \cup A$. If it is not, then $X$ must be independent. So we can suppose that $X \in S \cup N(S) \cup A$.

	If $X \in S\cup N(S)$, then $X$ is dependent if and only if it is covered by some flat in $\mathcal{Z}$.

	Hence, for each $r$-set that is not in $A$, we can reconstruct whether it is dependent from $(S,\mathcal{Z})$ alone. It remains to identify the non-bases in $A$. By Claim~\ref{claim:reconstruction}, we only need to verify, for each $X \in A$, if at least one of (i)--(iv) holds. Verification of each of these items depends only on $(S,\mathcal{Z},T)$:
	\begin{enumerate}[(i)]
		\begin{item}
			By Claim~\ref{claim:S-determines-A}, the set $A$ is completely determined by $S$, so for each neighbour of $X$, checking if it belongs to $A$ can be done if only $S$ is available. If $Y$ is a neighbour of $X$ that is not in $A$, then either it is not in $S\cup N(S)$, in which case $Y$ must be independent -- or it is in $S \cup N(S)$, in which case dependency of $Y$ depends only on $(S,\mathcal{Z})$.
		\end{item}
		\begin{item}
			Similar.
		\end{item}
		\begin{item}
			$X \in T$ obviously depends only on $T$.
		\end{item}
		\begin{item}
			For each neighbour $Y$ of $X$, $Y \in T$ depends only on $T$. By Claim~\ref{claim:S-determines-A}, for each neighbour of $Y$, determining whether it is in $A$ depends only on $S$. If the neighbour is not contained in $A$, then we can deduce from $(S,\mathcal{Z})$ whether it is independent or not, by the previous part of this proof. \qedhere
		\end{item}
	\end{enumerate}
\end{proof}

\subsection{Bounding the number of matroids}
Let us write $m'_n$ (resp.\ $m'_{n,r}$) for the number of matroids (resp.\ rank-$r$ matroids) on groundset $[n]$ that do not contain loops or coloops.
\begin{lemma}\label{lemma:mnr_bound}
	$m'_{n,r}
		\le s_{n,r} \sum_{k=0}^{2\lceil \tilde\sigma_{n,r} N\rceil} \binom{2^n(n+1)}{k}
	$.
\end{lemma}
\begin{proof}
In view of Lemma~\ref{lemma:concise-description}, a matroid $M \in \MM_{n,r}$ without loops or coloops can be described by a pair $(U, \mathcal{Z})$, in which  $U$ is a stable set of $J(n,r)$, and $\mathcal{Z} \subseteq \{(F,r_M(F)) : F \in \mathcal{F}(M)\}$. By Claim~\ref{claim:output_bound}, we can assume that $|\mathcal{Z}| \le 2\lceil \tilde\sigma_{n,r} N\rceil$.

Note that $m'_{n,r}$ is at most the number of pairs $(U, \mathcal{Z})$, and a bound on the number of such pairs follows immediately from the following two observations. First, as $U$ is a stable set in $J(n,r)$, it can be chosen in at most $i(J(n,r)) = s_{n,r}$ ways. Second, as $\mathcal{Z}$ is a subset of $2^{[n]} \times \{0,1,\ldots,n\}$ of cardinality at most $2\lceil \tilde\sigma_{n,r} N\rceil$, it can be chosen in at most $\sum_{k=0}^{2\lceil \tilde\sigma_{n,r} N\rceil} \binom{2^n (n+1)}{k}$ ways.
\end{proof}

The sum of binomial coefficients in the statement of Lemma~\ref{lemma:mnr_bound} can be bounded by an application of~\eqref{eq:binomial_sum}
\begin{equation}\label{eq:mnr_bound_nonuniform}
	\sum_{k=0}^{2\lceil \tilde\sigma_{n,r} N\rceil} \binom{2^n(n+1)}{k}
		\le \left(\frac{\e 2^n (n+1)}{2\left\lceil\tilde{\sigma}_{n,r} \binom{n}{r}\right\rceil}\right)^{2\left\lceil \tilde{\sigma}_{n,r} \binom{n}{r}\right\rceil}
		\le \left(\frac{\e 2^n (n+1)^3}{4\binom{n}{r}}\right)^{2 \left\lceil\tilde\sigma_{n,r} \binom{n}{r}\right\rceil}.
\end{equation}
%
The next lemma gives a bound that is uniform in $r$.
\begin{lemma}\label{lemma:sigma_bound}
	There exists $c > 0$ such that for sufficiently large $n$, and $0 \le r \le n/2$,
	\begin{equation*}
		\sum_{k=0}^{2\lceil \tilde\sigma_{n,r} N\rceil} \binom{2^n(n+1)}{k} \le \exp_2\left(c\frac{\log^2 n}{n^2}\binom{n}{\lfloor n/2\rfloor}\right).
	\end{equation*}
\end{lemma}
\begin{proof}
	We may assume that $n \ge 4$. Recall that $\tilde{\sigma}_{n,r} = \sigma_{n,r} + \frac{1}{r+1}\alpha_{n,r}$. It follows from~\cite[Lemma 19]{BansalPendavinghVanderpol2014} that $\sigma_{n,r}\binom{n}{r} \le \frac{8 \ln n}{n^2} \binom{n}{\lfloor n/2\rfloor}$. It is easily verified that $\frac{1}{r+1}\alpha_{n,r} \binom{n}{r}$ is increasing in $r=0,1,\ldots,\lfloor n/2\rfloor$, so
	\begin{equation*}
		\frac{1}{r+1}\alpha_{n,r} \binom{n}{r} \le \frac{4}{n^2} \binom{n}{\lfloor n/2\rfloor} \le \frac{3 \ln n}{n^2} \binom{n}{\lfloor n/2\rfloor}.
	\end{equation*}
	Combining these bounds, we obtain $\tilde\sigma_{n,r} \binom{n}{r} \le \frac{11 \ln n}{n^2} \binom{n}{\lfloor n/2\rfloor}$. An application of~\eqref{eq:binomial_sum} gives
	\begin{equation*}
		\max_{0 \le r \le n/2} \sum_{k=0}^{2\lceil \tilde\sigma_{n,r} N\rceil} \binom{2^n(n+1)}{k} \le \exp_2\left(\frac{22 \ln n}{n^2}\binom{n}{\lfloor n/2\rfloor} \log \left(\frac{\e 2^n (n+1)}{\frac{22 \ln n}{n^2} \binom{n}{\lfloor n/2\rfloor}}\right)(1+o(1))\right).
	\end{equation*}
	The lemma follows as the expression in the logarithm is $n^{O(1)}$ by~\eqref{eq:binomial_central}.
\end{proof}

\ignore{The following lemma bounds $\tilde\sigma N$ uniformly in $r$.
\begin{lemma}\label{lemma:sigma_bound_old}
	If $n \ge 4$ and $0 \le r \le n/2$, then $\lceil \tilde\sigma N\rceil \le \frac{11 \ln n}{n^2} \binom{n}{\lfloor n/2\rfloor}$.
\end{lemma}
\begin{proof}
	Recall that $\tilde\sigma_{n,r} \binom{n}{r} = \left(\sigma_{n,r} + \frac{1}{r+1}\alpha_{n,r}\right)\binom{n}{r}$. It follows from~\cite[Lemma 19]{BansalPendavinghVanderpol2014} that $\sigma_{n,r} \le \frac{8 \ln n}{n^2} \binom{n}{\lfloor n/2\rfloor}$. It is easily verified that $\frac{1}{r+1}\alpha_{n,r} \binom{n}{r}$ is increasing in $r=0,1,\ldots,\lfloor n/2\rfloor$, so
	\begin{equation*}
		\frac{1}{r+1}\alpha_{n,r} \binom{n}{r} \le \frac{4}{n^2} \binom{n}{\lfloor n/2\rfloor} \le \frac{3 \ln n}{n^2} \binom{n}{\lfloor n/2\rfloor},
	\end{equation*}
	from which the claim follows.
\end{proof}
}      

\begin{theorem}
	$\log m'_n \le (1+o(1))\log s_n$ as $n \to \infty$.
\end{theorem}
\begin{proof}
	Upon combining the upper bound on $m'_{n,r}$ with the fact that $s_{n,r} \le s_n$, we find by Lemma~\ref{lemma:sigma_bound} that
	\begin{equation*}
		m'_{n,r} \le s_n \exp_2\left(c \frac{\log^2 n}{n^2} \binom{n}{\lfloor n/2\rfloor}\right),
	\end{equation*}
	at least for $r \le n/2$. By duality, this same bound holds for $m'_{n,n-r}$. As $m'_n = \sum_r m'_{n,r}$, we obtain
	\begin{equation*}
		\log m'_n \le \log(s_n) + \log(n+1) + c\frac{\log^2 n}{n^2} \binom{n}{\lfloor n/2\rfloor}.
	\end{equation*}
	The theorem now follows since $\log s_n \ge \frac{1}{n} \binom{n}{\lfloor n/2\rfloor}$.
\end{proof}
It was shown in~\cite[Theorem 2.3]{MayhewNewmanWelshWhittle2011} that almost every matroid has no loops or coloops, i.e.\ that
\begin{equation}\label{eq:almost-no-loops}
	m'_n = m_n(1-o(1)) \qquad\text{as $n \to \infty$}.
\end{equation}
This immediately implies the following corollary.
\begin{corollary}\label{crl:main_result}
	$\log m_n = (1+o(1)) \log s_n$ as $n \to \infty$.
\end{corollary}

\subsection{The rank of a typical matroid}
It was shown in~\cite[Corollary 2.3]{LowranceOxleySempleWelsh2013} that for all $\varepsilon > 0$, the fraction of matroids having rank $r$ in the range $(1/2 - \varepsilon)n < r < (1/2 + \varepsilon)n$ tends to~1. In this section, we prove a slightly stronger statement.
\begin{theorem}\label{thm:rank}
	There exists $\beta_0$ such that for $\beta > \beta_0$ the fraction of matroids with $n/2 - \beta\sqrt{n} < r < n/2 + \beta\sqrt{n}$ tends to 1 as $n \to \infty$.
\end{theorem}
\begin{proof}
	By duality, it suffices to show that $r(M) \le n/2 - \beta\sqrt{n}$ for a vanishing fraction of matroids. In view of~\eqref{eq:almost-no-loops}, it suffices to restrict our attention to matroids without loops or coloops. In fact, we will show that for $\beta$ sufficiently large
	\begin{equation*}
		\frac{1}{m'_n} \sum_{r=0}^{\lfloor n/2 - \beta\sqrt{n}\rfloor} m'_{n,r} \to 0\qquad\text{as $n \to \infty$}.
	\end{equation*}
	The term $m'_{n,r}$ can be bounded by combining Lemma~\ref{lemma:mnr_bound} with the upper bound on $s_{n,r} = i(J(n,r))$ provided by Theorem~\ref{thm:stablesets}. As $m'_n \ge s_{n,\lfloor n/2\rfloor} \ge 2^{\frac{1}{n}\binom{n}{\lfloor n/2\rfloor}}$, it follows that
	\begin{equation*}
		\begin{split}
			\frac{1}{m'_n} \sum_{r=0}^{\lfloor n/2 - \beta\sqrt{n}\rfloor} m'_{n,r}
				&\le 2^{-\frac{1}{n}\binom{n}{\lfloor n/2\rfloor}} \sum_{r=0}^{\lfloor n/2 - \beta\sqrt{n}\rfloor}
					\left[\sum_{s=0}^{\lceil \sigma_{n,r} \binom{n}{r}\rceil} \binom{\binom{n}{r}}{s}
					2^{\alpha_{n,r} \binom{n}{r}}
					\sum_{k=0}^{2\lceil \tilde\sigma_{n,r} \binom{n}{r}\rceil} \binom{2^n(n+1)}{k}\right] \\
				&\le 2^{-\frac{1}{n}\binom{n}{\lfloor n/2\rfloor}}
					\sum_{r=0}^{\lfloor n/2 - \beta\sqrt{n}\rfloor} \left[ 2^{\alpha_{n,r} \binom{n}{r}}
					\left(\sum_{k=0}^{2\lceil \tilde\sigma_{n,r} \binom{n}{r}\rceil} \binom{2^n(n+1)}{k}\right)^2 \right],
		\end{split}
	\end{equation*}
	which, by Lemma~\ref{lemma:sigma_bound} and the inequality $\alpha_{n,r} \le \frac{2}{n}$ is at most
	\begin{multline*}
		2^{-\frac{1}{n}\binom{n}{\lfloor n/2\rfloor}} \sum_{r=0}^{\lfloor n/2 -\beta\sqrt{n}\rfloor} \exp_2\left(\frac{2}{n}\binom{n}{r} + 2c\frac{\log^2 n}{n^2}\binom{n}{\lfloor n/2\rfloor}\right) \\
			\le n \exp_2\left(-\frac{1}{n} \binom{n}{\lfloor n/2\rfloor} (1-o(1)) + \frac{2}{n}\binom{n}{\lfloor n/2 - \beta\sqrt{n}\rfloor}\right) \\
			= n\exp_2\left(-\frac{1}{n}\binom{n}{\lfloor n/2\rfloor} \left(1 - 2\e^{-2\beta^2})(1-o(1))\right)\right),
	\end{multline*}
	where the final equality follows from Lemma~\ref{lemma:binomial4} with $k = \beta\sqrt{n}$. For sufficiently large $\beta$, the right-hand side tends to 0, thus concluding the proof.
\end{proof}
	We like to remark at this point that Theorem~\ref{thm:rank} can be proved using the bound on $m_{n,r}$ that was derived in the proof of~\cite[Theorem 3]{BansalPendavinghVanderpol2014}. But since that paper does not have a separate lemma which we can refer to here, we make use of Lemma~\ref{lemma:mnr_bound} of the present paper.

\section{Some remaining problems}
\subsection{Counting the matroids without circuit-hyperplanes}
It was conjectured in ~\cite[Conjecture 22]{BansalPendavinghVanderpol2014} that
\begin{equation}\lim_{n\rightarrow\infty} \frac{\#\{M\in \MM_n: M\text{ has no circuit-hyperplanes }\}}{m_{n}}=0\end{equation}
We have tried to prove this conjecture by analysing the behaviour of (variants of) our compression algorithm on matroids without circuit-hyperplanes, so far without any success. We like to encourage the reader to give it another try, since it does feel as if we are very close.  The key to proving a sufficient bound seems to be the behaviour of the algorithm when picking $T\subseteq A$. A more intelligent decoder may be able to reconstruct the matroid from a much sparser set $T$. Note that the asymptotic upper bounds on $|S|$ and $|\ZZ|$ do not get worse essentially if we insist that the algorithm continues its main loop while $\Delta(G[A])\geq \epsilon r$, for any fixed $\epsilon>0$. 

\subsection{The rank of a typical matroid}
It was conjectured by Mayhew, Newman, Welsh, and Whittle in~\cite[Conjecture 1.10]{MayhewNewmanWelshWhittle2011} that asymptotically almost all matroids have rank between $(n-1)/2$ and $(n+1)/2$. By Theorem~\ref{thm:rank}, almost all matroids on $n$ elements  have a rank between $n/2-\beta\sqrt{n}$ and $n/2-\beta\sqrt{n}$. At the heart of the argument lies the fact that for sufficiently large $k$ the ratio $m_{n, \lfloor n/2\rfloor-k}/m_{n,\lfloor n/2\rfloor}$ tends to 0, using that
\begin{equation*}
	\frac{m_{n, \lfloor n/2\rfloor-k}}{m_{n,\lfloor n/2\rfloor}}\leq \frac{s_{n, \lfloor n/2\rfloor-k}}{s_{n,\lfloor n/2\rfloor}} \cdot 2^{O\left(\frac{\log^2 n}{n^2}\binom{n}{\lfloor n/2\rfloor -k}\right)}
\end{equation*}
We see no better way to bound the factor $s_{n, n/2-k} /s_{n,n/2}$ than by combining the lower bound of Graham and Sloane and the upper bound from ~\cite{BansalPendavinghVanderpol2014}, and we ask if a direct comparison would perhaps be possible. It would be a good start if we could argue that asymptotically all {\em sparse paving} matroids on $n$ elements have a rank between $(n-1)/2$ and $(n+1)/2$. Presently we cannot even prove the unimodality of $s_{n,r}$, i.e. that $s_{n,r}<s_{n,r'}$ for all $0<r<r'\leq n/2$. 

\subsection{Comparing $\log s_{n,r}$ and $\log m_{n,r}$ for small $r$}

By Theorem~\ref{thm:rank}, most matroids have rank close to $n/2$. Consequently, in the derivation of our main result, Corollary~\ref{crl:main_result}, we are mainly interested in comparing $m_{n,r}$ to $s_{n,r}$ for $r \approx n/2$. In this regime, the upper bound
\begin{equation*}
	\log m_{n,r}
		\le \log s_{n,r}+ \log \sum_{k=0}^{2\lceil \tilde\sigma_{n,r} N\rceil} \binom{2^n(n+1)}{k}
		\le \log s_{n,r}+ 2\left\lceil\tilde{\sigma}_{n,r} \binom{n}{r}\right\rceil \log\left(\frac{\e 2^n (n+1)^3}{4\binom{n}{r}}\right)
\end{equation*}
from Lemma~\ref{lemma:mnr_bound} combined with~\eqref{eq:mnr_bound_nonuniform} suffices to show that $(\log s_{n,r})/(\log m_{n,r})\rightarrow 1$ as $n\rightarrow \infty$.

On the other hand, when $r$ (or by duality $n-r$) is small, this upper bound will not suffice, as illustrated by the following result.

\begin{lemma}
	There exists $c > 0$ such that $\log \sum_{k=0}^{2\lceil \tilde\sigma_{n,r} N\rceil} \binom{2^n(n+1)}{k} \ge \binom{n}{r}$ for all $r \le c \log n$ and $n$ sufficiently large.
\end{lemma}
\ignore{
\begin{proof}
Throughout the argument, we will assume that $n$ is sufficiently large to support our statements. For convenience, we will use the bounds $\frac{\ln(r(n-r))}{r(n-r+1)} \le \tilde\sigma_{n,r} \le \frac{2\ln(r(n-r))}{r(n-r+1)}$. We bound the sum of binomial coefficient by its largest term, and an application of~\eqref{eq:binomial1} to obtain
\begin{equation*}
	\log \sum_{k=0}^{2\lceil \tilde\sigma_{n,r} N\rceil} \binom{2^n (n+1)}{k}
		\ge \log \left(\frac{2^n(n+1)}{2\lceil \tilde\sigma_{n,r} \binom{n}{r}\rceil}\right)^{2\lceil \tilde\sigma_{n,r} N\rceil}
		\ge 2\frac{\ln(r(n-r))}{r(n-r+1)} N \log \left(\frac{2^n(n+1)}{8 \frac{\ln(r(n-r))}{r(n-r+1)} \binom{n}{r}}\right).
\end{equation*}
In order to obtain a lower bound, we bound the binomial coefficient inside the logarithm by
\begin{equation}\label{eq:binomial_entropy}
	\binom{n}{r} \le 2^{nH(r/n)} \le 2^{2\frac{r}{n}\log n}, \qquad 1 \le r \le n/2,
\end{equation}
where $H(p) = -p \log p - (1-p) \log (1-p)$ is the binary entropy function. Using~\eqref{eq:binomial_entropy},
\begin{equation*}
	\frac{2^n (n+1)}{4\tilde\sigma_{n,r} \binom{n}{r}} \ge 2^{n\left(1-2\frac{r}{n}\log n\right)}.
\end{equation*}
Hence,
\begin{multline*}
	2\frac{\ln(r(n-r))}{r(n-r+1)} N \log \left(\frac{2^n(n+1)}{8 \frac{\ln(r(n-r))}{r(n-r+1)} \binom{n}{r}}\right) \\
		\ge 2\frac{\ln(r(n-r)+1)+1}{r(n-r+1)} n\left(1 - \frac{r}{n}\log n\right) \binom{n}{r}
		\ge \frac{\ln n}{r} \binom{n}{r}
\end{multline*}
when $r \le c \log n$. The claim follows.
\end{proof}
}
We note at this point that the best asymptotic lower bound on $s_{n,r}$ for fixed $r$ is strictly better than the general bound of $\log s_{n,r}\geq \alpha(J(n,r))\geq \binom{n}{r}/n$. The better bound follows from the work of Keevash~\cite[Thm 6.8]{Keevash2014} who has recently proved an asymptotic estimate of the number of Steiner systems with parameters $(n,r,q)$, where~$r$ and~$q$ are fixed:
\begin{equation*}
	\log S(n,r,q)=(1+o(1))\binom{q}{r}^{-1}\binom{n}{r}(q-r)\log n\qquad\text{ as }n\rightarrow \infty.
\end{equation*}  
Since each Steiner system with parameters $(n, r-1, r)$ uniquely determines a stable set of the Johnson graph $J(n,r)$, Keevash's estimate gives an asymptotic lower bound of 
\begin{equation}\label{eq:lowerbound_keevash}
	\log s_{n,r}\geq \log S(n,r-1,r)= (1+o(1))r^{-1}\binom{n}{r-1}\log n\qquad\text{ as }n\rightarrow \infty.
\end{equation}
The lower bound~\eqref{eq:lowerbound_keevash} is tight up to the $(1+o(1))$-factor. This follows from the trivial upper bound
\begin{equation*}
	\log s_{n,r} = \log i(J(n,r)) \le \log \sum_{k=0}^{\alpha_{n,r}\binom{n}{r}} \binom{\binom{n}{r}}{k} = (1+o(1))\alpha_{n,r} \binom{n}{r} \log n \qquad\text{ as }n\rightarrow \infty
\end{equation*}
and the identity $r^{-1} \binom{n}{r-1} = \alpha_{n,r}\binom{n}{r}$.

We speculate that for fixed $r$, the ratio $s_{n,r}/m_{n,r}$ tends to 0  as $n\rightarrow \infty$, while at the same time, $(\log s_{n,r})/(\log m_{n,r})$ tends to 1.

\subsection{Comparing $s_n$ and $m_n$}

The matroid encoding procedure described in this paper returns a pair~$(U,\ZZ)$, where~$U$ is a stable set and~$\ZZ$ is a partial flat cover. Note that we could have been more economical in constructing~$\ZZ$. For example, when the encoding procedure is run on a sparse paving matroid, that matroid can be reconstructed from $U$ alone. This suggests that in many cases~$\ZZ$ can be pruned, while the original matroid can still be recovered.

The following information-theoretic perspective may be useful.
Suppose that~$(U,\ZZ)$ is chosen uniformly at random from the set of all possible outputs of the encoding procedure when its input is some matroid in $\MM_{n,r}$ without loops or coloops.
By the chain rule for the entropy function $\HH{\cdot}$,
\begin{equation*}
	\log m'_{n,r} = \HH{U,\ZZ} = \HH{U} + \HH{\ZZ \mid U}.
\end{equation*}
The term $\HH{U}$ is at most $\log s_{n,r}$. Currently, we naievely bound $\HH{\ZZ \mid U}$ by the logarithm of the number of possible partial flat covers, which does not take into account the mutual information between~$U$ and~$\ZZ$. We expect that $\HH{\ZZ \mid U}$ is much smaller than our naieve bound, but presently do not see a more careful analysis of this quantity.

\bibliographystyle{alpha}
\bibliography{bib}
\end{document}